\documentclass[11pt]{amsart}

\calclayout

\usepackage{color}
\usepackage{amsfonts}
\usepackage{amssymb}
\usepackage{amsmath}
\usepackage{hyperref}
\usepackage{multicol}
\usepackage{lscape}

\usepackage{graphicx}
\usepackage{enumitem}
\usepackage{subcaption}
\usepackage{wrapfig}

\newcommand{\bI}{\mathbb I}

\newcommand{\al}{\alpha}

\newcommand{\vf}{\varphi}
\newcommand{\om}{\omega}
\newcommand{\Om}{\Omega}

\newcommand{\cD}{\mathcal D}

\newcommand{\cF}{\mathcal F}

\newcommand{\cR}{\mathcal R}
\newcommand{\cT}{\mathcal T}

\newtheorem{theorem}{Theorem}[section]

\newtheorem{lemma}[theorem]{Lemma}
\newtheorem{prop}[theorem]{Proposition}
\theoremstyle{definition}
\newtheorem{remark}[theorem]{Remark}
\newtheorem{defin}[theorem]{Definition}

\numberwithin{equation}{section}

\newcounter{vremennyj}

\begin{document}

\title[]%
{
 A comparison of box and Carleson conditions on  bi-trees
}
\author[I.~ Holmes]{Irina Holmes}
\thanks{IH is partially supported by the NSF as an NSF Postdoc under Award No.1606270}
\address{Department of Mathematics, Michigan Sate University, East Lansing, MI. 48823}
\email{holmesir@msu.edu}
\author[G.~ Psaromiligkos]{Georgios Psaromiligkos}
\address{Department of Mathematics, Michigan Sate University, East Lansing, MI. 48823}
\email{psaromil@msu.edu}
\author[A.~Volberg]{Alexander Volberg}
\thanks{AV is partially supported by the NSF grant DMS-160065}
\address{Department of Mathematics, Michigan Sate University, East Lansing, MI. 48823}
\email{volberg@math.msu.edu \textrm{(A.\ Volberg)}}
\makeatletter
\@namedef{subjclassname@2010}{
  \textup{2010} Mathematics Subject Classification}
\makeatother
%
%
\begin{abstract}
In this note  we give an examples of measures satisfying the box condition on certain sub-bi-trees (see below) but not satisfying Carleson condition on those sub-bi-trees. This can be considered as a certain counterexample for two weight bi-parameter embedding of Carleson type.
Our type of counterexample is impossible for a simple tree. In the case of a simple tree, the box condition, Carleson condition and two weight embedding are all equivalent. In the last section we show that bi-parameter box condition  implies the bi-parameter capacitary estimate for 
dyadic rectangles.

\end{abstract}
\maketitle

\section{The statement of the problem}
\label{st}

\subsection{Two-parameter embeddings}
Let $Q$ be the unit square $[0,1]^2$. We consider all its dyadic sub-squares  of size $2^{-N}\times 2^{-N}$, and denote an individual such square by the symbol $\om$.

On each $\om$ we are given two non-negative numbers: $\vf_\om$ and $\mu(\om)$. The collection $\{\mu(\om)\}_{\om}$ will be a measure on $Q$. This means that we think that
$2^{2N}\mu(\om)$ is the constant (on $\om$) density with respect to Lebesgue measure $m_2$.

We call a rectangle $R\subset Q$ \textit{$N$-coarse} if it is a) a dyadic rectangle, b) is a union of $\om$'s.  We call a set $U\subset Q$ $N$-coarse if it  is a union of $\om$'s. Our convention will be that all dyadic rectangles below -- with no exception -- are $N$-coarse. The number $N$ will be very large, and we will need all estimates to be {\it independent} of $N$.

We define the measure on $Q$, which is calculated on each $N$-coarse subset $U\subset Q$ as follows
$$
\mu(U) =\sum_{\om\subset U} \mu(\om)\,.
$$
Here $N$-coarse subsets are  subsets that are unions  of $\om$'s.
In particular, for any dyadic rectangle inside $Q$ we have
$$
\mu(R) =\sum_{\om\subset R} \mu(\om)\,.
$$
Shortly, the collection $\{\mu(\om)\}_{\om}$ will be called  a measure $\mu$ on $Q$.

\bigskip

Let $\cD(Q)= \cD_N(Q)$ be the family of all ($N$-coarse, we mention this for the last time) dyadic sub-rectangles of $Q$.
Similarly, we use $\cD(R)$ to  the family of all dyadic sub-rectangles of $R$, $R\in \cD(Q)$.

We think about the collection $\{\vf_\om\}_\om$ as a step function on $Q$, constant on every $\om$ and equal to $\vf_\om$ on every $\om$. So we use the notation $\vf= \sum_{\om} \vf_\om\cdot  {\bf 1}_\om$, and we can integrate $\vf$ with respect to $\mu$: let $R\in \cD$, then
$$
\int_R \vf\, d\mu =\sum_{\om\subset R} \vf_\om \mu(\om)\,.
$$

\bigskip

We are interested in several embedding theorems. 

\medskip

Embedding theorems of Carleson type serve as the simplest ``singular" integral operators in many situations. Here these will be two-parameter embeddings of Carleson type.


We are interested to understand the following $(\cD, \mu)$ embedding:
\begin{equation}
\label{ALe}
\sum_{R\in \cD} (\int_R \vf\, d\mu)^2\, \alpha_R \le C \int_Q \vf^2\, d\mu\,,
\end{equation}
where $\al:=\{\al_R\}_{R\in \cD}$ is a family of non-negative numbers. Family $\al$ can be considered as the second weight. It sits on the bi-tree $\cT_\cD$, and the embedding is considered as an operator
$$
L^2(Q, \mu)\to \ell^2(\cT_\cD, \al)\,.
$$
So we are dealing with a two weight two-parameter embedding. Let us remind the reader that  two weight one-parameter embedding is very well studied, important, and for it everything is clear.

\subsection{One-parameter two weight embedding}
\label{1param}

We start with $L^2(I_0, \mu)$. and want to imbed it into $\ell^2(T, \al)$, where $T$ is the usual dyadic tree, and $I_0$ is the interval $[0,1]$. The role of dyadic rectangles is played by dyadic sub-intervals of $[0,1]$, and we will call such intervals by letter $I$. So the embedding to a simple tree (one-parameter embedding) looks like
\begin{equation}
\label{1PEmb}
\sum_{I\in \cD_1} (\int_I \vf\, d\mu)^2\, \alpha_I \le C \int_{I_0} \vf^2\, d\mu\,,
\end{equation}
where $\al:=\{\al_I\}_{I\in \cD_1}$ is a family of non-negative numbers, and $\cD_1$  means the collection of dyadic subintervals of $I_0$.

It is well known when \eqref{1PEmb} happens, see, e. g. \cite{NTV99}, \cite{AHMV}.
\begin{theorem}
\label{1PEmbTh}
For embedding \eqref{1PEmb} to happen it is necessary and sufficient to have  the following box condition:
\begin{equation}
\label{1Pbox}
\forall I\in \cD_1, \quad \sum_{J\in \cD_1, J\subset I} \mu(J)^2\, \alpha_J \le C_0 \,\mu(I)\,,
\end{equation}
\end{theorem} 

The necessity is, of course, trivial: just test \eqref{1PEmb} on $\vf={\bf 1}_I$. The sufficiency is not trivial. It is morally equivalent to a two weight maximal function result, see \cite{Saw}. It has many proofs. One of the simplest involves the Bellman function technique, see e. g. \cite{NTV99}, \cite{AHMV}.

\medskip

Our main goal in this note is to show that such a result fails in a bi-parameter case. However, we 
do not know (sometimes we do, but not in general) if one can replace box condition by more complicated Carleson condition in bi-parameter (= two parameter) embedding. We do know, however, that in the case when $\al_R\equiv 1$ the bi-parameter embedding is equivalent to Carleson condition \cite{AMPS}, \cite{AHMV}.

\subsection{Formulation of the question and formulation of the result}
\label{formu}

In a bi-parameter case we have the analog of box condition  formulated in terms of dyadic rectangles (=boxes)
\begin{equation}
\label{2Pbox}
\forall R_0\in \cD, \quad \sum_{R\in \cD, R\subset R_0} \mu(R)^2\, \alpha_R \le C_0\, \mu(R_0)\,,
\end{equation}

Of course this condition is necessary for bi-parameter embedding  \eqref{ALe}. But we will show below that, unlike the one parameter case, it is not sufficient for  bi-parameter embedding  \eqref{ALe}. 

But there is a more powerful necessary condition in the bi-parameter case. We call it Carleson condition.
Let symbol $\Om$ always denote a finite union of  some rectangles  from $\cD$. Here is the Carleson condition:
\begin{equation}
\label{2PC}
\forall \Om \subset Q, \quad \sum_{R\in \cD, R\subset \Om} \mu(R)^2\, \alpha_R \le C_0\, \mu(\Om)\,,
\end{equation}

Bi-parameter Carleson condition \eqref{2PC} is obviously necessary for the embedding \eqref{ALe}. It is enough to test \eqref{ALe} on $\vf={\bf 1}_\Om$. 

But is it sufficient? For the general case of $\al$ we do not know. But for $\al$ identically $1$, the answer is ``yes", and it is a difficult result from \cite{AMPS}, \cite{AHMV}.

\begin{remark}
We will show below that for certain $\al$  one can have bi-parameter box condition\eqref{2Pbox} to hold, but bi-parameter Carleson condition \eqref{2PC}  fails. One such example is easy to derive from famous Carleson's counterexample \cite{Carleson}. We will note that. Then we will build yet another counterexample of the type: bi-parameter box condition holds, but embedding does not hold. Unfortunately in these both counterexamples -- the one derived from Carleson's one, and another one that we build -- the sequence $\{\al_R\}_{R\in \cD}$ is very wild. The example that one obtains from \cite{Carleson}  has $\{\al_R\}_{R\in \cD}$  even ``unbounded". In both examples $\al_R$ must be $0$ quite often.
\end{remark}

\begin{remark}
We, so far, cannot construct  an example with $\al_R\equiv 1$ and such that the bi-parameter box condition holds, but the embedding does not hold.
\end{remark}

\medskip

\noindent{\bf Acknowledgement.} We are grateful to Fedor Nazarov for fruitful conversation which allowed us to improve and/or complement certain results below.

\subsection{Particular cases and more questions}
Let family $\al$ consists only of $1$ and $0$. And let $\cF\subset \cD$ denote the support of $\al$. Then the inequality of  our main interest, namely, \eqref{ALe}, becomes
\begin{equation}
\label{Fe}
\sum_{R\in \cF} (\int_R \vf\, d\mu)^2\, \le C \int_Q \vf^2\, d\mu\,,
\end{equation}

We wish to consider sub-problems, where $\cF$ (or $\cD\setminus \cF$) has special structure.
The most interesting case is 
\begin{equation}
\label{DF}
\cF=\cD\,.
\end{equation}

\begin{figure}
\centering
\includegraphics[scale=.35]{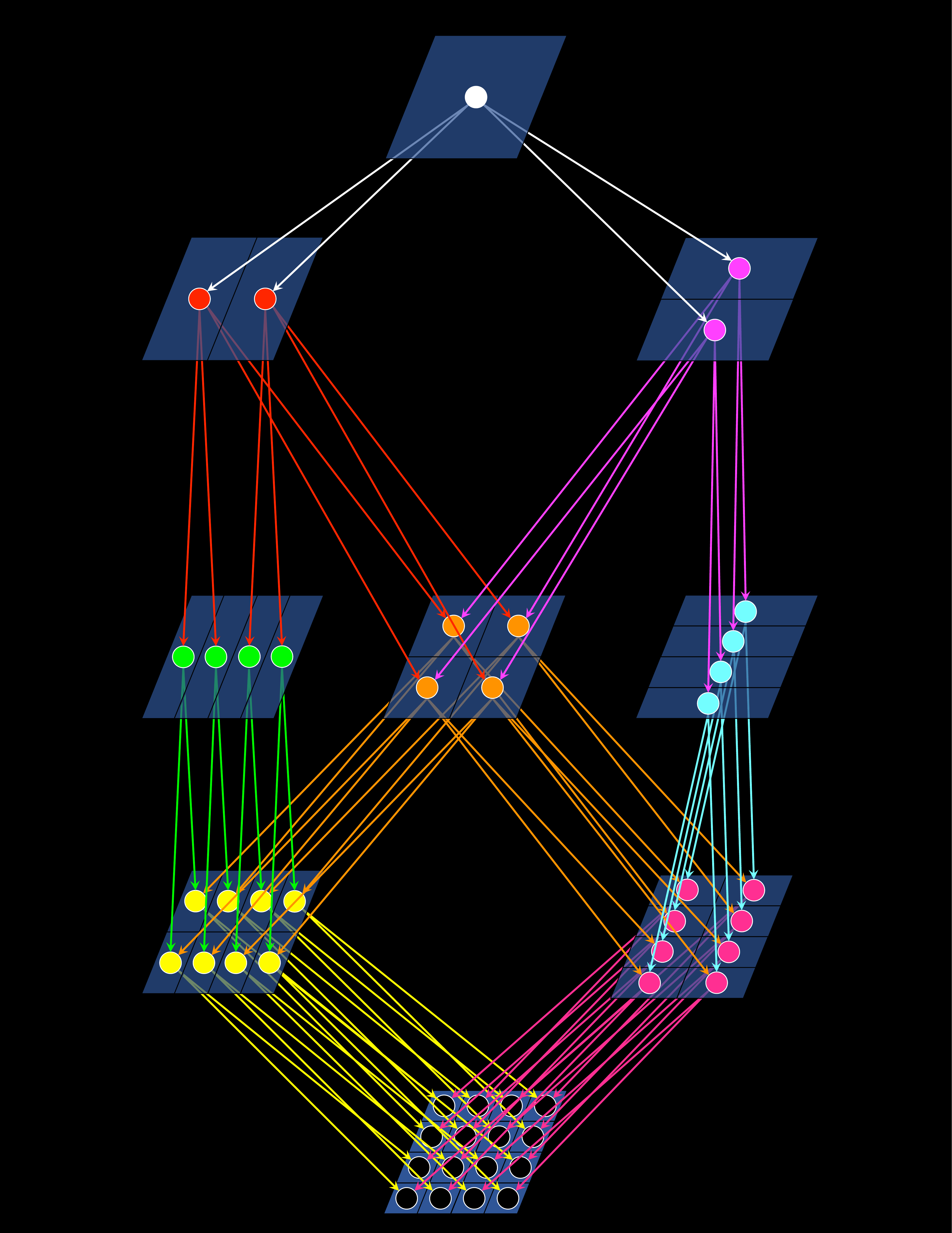}
\caption{Bi-tree for $N=2$.}
\label{fig:bitree}
\end{figure}

As we already mentioned above, in this case the necessary and sufficient condition for 
\begin{equation}
\label{De}
\sum_{R\in \cD} (\int_R \vf\, d\mu)^2 \le C \int_Q \vf^2 \,d\mu\,.
\end{equation}
was found in \cite{AMPS} and then in \cite{AHMV}.
This condition (it is proved in \cite{AMPS}, \cite{AHMV} that it is equivalent to \eqref{De}) is the following:
let $\cD'$ is any sub-family of $\cD$, and $U=\cup_{R_0\in \cD'} R_0$, then let us  have
\begin{equation}
\label{CDe}
\forall \,U,\quad  \sum_{R\subset U, R\in \cD} ( \mu(R))^2 \le C' \mu(U)\,.
\end{equation}
It is trivial to see that \eqref{De} implies \eqref{CDe}: just consider $\vf= {\bf 1}_U$ as a test function in \eqref{De}. The opposite implication is of quite high level of non-triviality (it looks like a $T1$ theorem, of course): see \cite{AMPS}, \cite{AHMV}.

\begin{defin}
Condition \eqref{CDe} is called {\it  bi-tree Carleson condition} or simply {\it  Carleson condition} .
\end{defin}

\bigskip

If $\cF\neq \cD$ the problem is still interesting, especially if $\cF$ is rich enough, but $\cD\setminus \cF$ is not trivial (if it is just one rectangle, it is the same problem as before).

\medskip

So let us consider a particular case when $\cF$ is natural and interesting, and another case when $\cD\setminus \cF$  has an interesting structure.

\medskip

We would like to use the language of {\bf bi-tree}. Let every $R\in \cD$ be viewed as a vertex, and the edge goes from $R_1$ to $R_2$ if and only if $R_2\subset R_1, R_1, R_2\in \cD$, and $R_1$ is a parent of $R_2$. Any such $R_1$ as above  is called a parent of $R_2$ if $R_2$ is obtained by splitting $R_1$ vertically or horizontally in two equal halves, any such $R_2$ as above is called a child of $R_1$. A child may have two parents.

  Then $\cD$ becomes a finite bi-tree (with $2^{2N}$ terminal vertices at each $\om$ and one ``root" corresponding to $Q$). We think about bi-trees growing down, the root is the highest vertex.
Let us call it $\cT_\cD$.

We say that a child is {\it below} a parent (so we have arrows, direction on the graph).  Bi-tree has many  non-directed cycles but does not have directed cycles.

Naturally we also have   $\cT_\cF$ for every $\cF\subset \cD$, and
$$
\cT_\cF\subset \cT_\cD\,.
$$

\begin{figure}
\centering
\includegraphics[scale=.2]{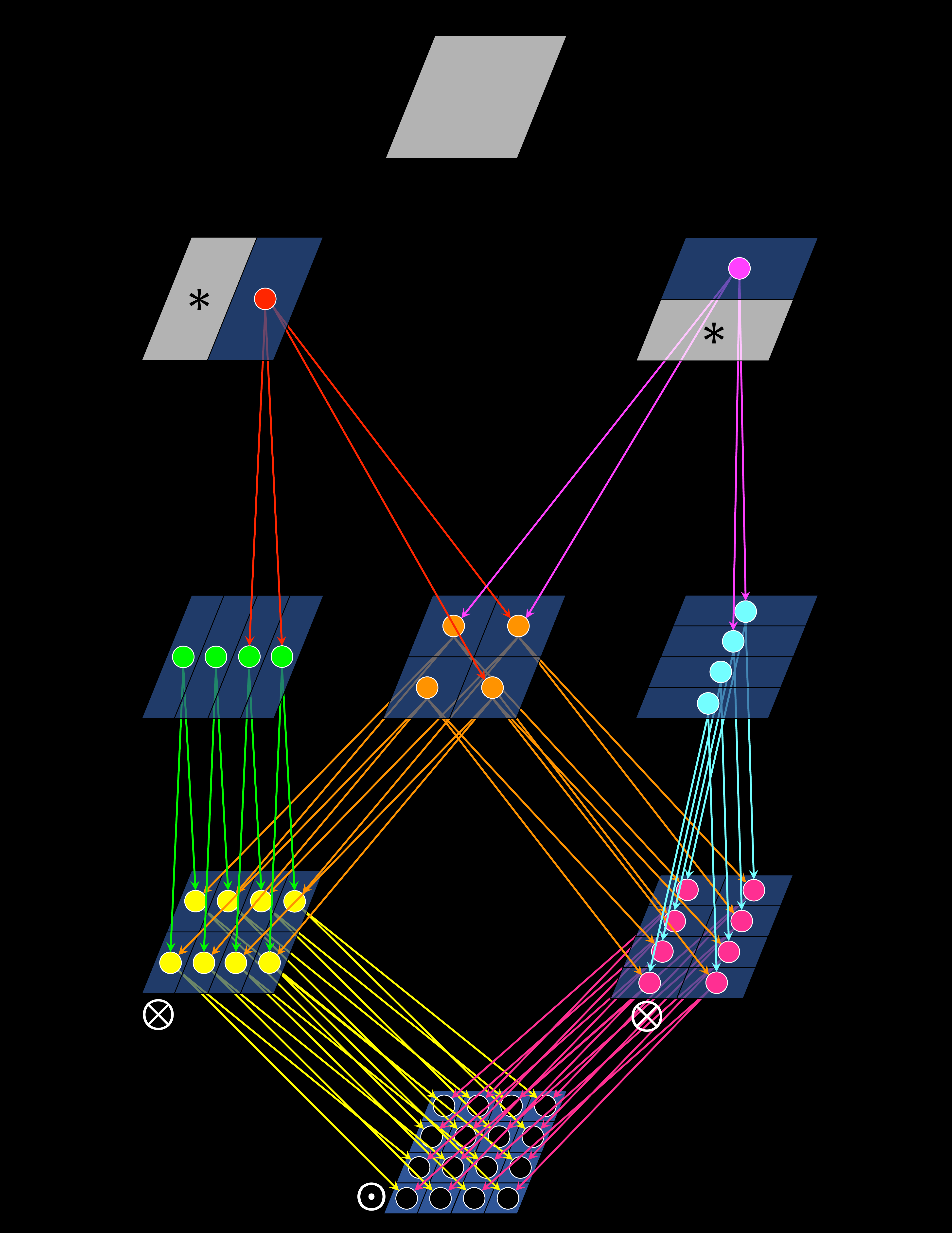}
\caption{Bi-tree ($N=2$) with ``hooked'' rectangles ''above the hyperbola'' deleted.
This is both pruned and cut. As a cut bi-tree, this is obtained by deleting white rectangles.
In terms of inequality \eqref{ALe} white rectangles $W$  correspond to $\al_W=0$. Coefficients $\al_R=1$ for all other rectangles.
}
\label{fig:bitree-nh}
\end{figure}

\begin{defin}
We call  $\cT_\cF$  a \textbf{pruned bi-tree} if  it is a connected subgraph of the full bi-tree $\cT_\cD$ such that all $\om \in \partial \cT_\cD$ (the smallest elements of  $\cT_\cD$) belong to $\cT_\cF$. By connected here we mean connected as an undirected graph. In literature such directed graphs are called weakly  connected.
\end{defin}

\begin{defin}
We call  $\cT_\cF$  a \textbf{cut bi-tree} if for every vertex $\al \in \cT_\cD\setminus  \cT_\cF$, any of its its parents also belong to $\cT_\cD\setminus \cT_\cF$. In other words, one obtains a cut bi-tree by deleting vertices $\alpha_1, \ldots, \alpha_k$ and then also deleting all parents of each $\alpha_i$. Notice that cut bi-trees have the root (the vertex corresponding to the unit square $Q$) cut out.
\end{defin}

\medskip

\noindent{\bf Example.} The reader will see below the definition of {\it hooked} rectangles. Let $\cF$ be all dyadic rectangles except the hooked ones. The corresponding bi-tree $\cT_\cF$ is a cut bi-tree. Also if one deletes not all hooked rectangles but only hooked rectangles with upper-right vertex ``above the hyperbola" (see below) the resulting family again generates a cut bi-tree. Notice that it is also a pruned tree. See Figure \ref{fig:bitree} and \ref{fig:bitree-nh}.

\bigskip

Here is an obvious proposition.
\begin{prop}
Let $\cF\subset \cD$. Then  embedding \eqref{Fe}, that is:
\begin{equation*}
\forall \vf,\quad  \sum_{R\in \cF} (\int_R \vf d\mu)^2 \le C \int_Q \vf^2 d\mu
\end{equation*}
implies that for any collection $\cF'\subset \cD$ and $U'=\cup_{R_0\in \cF'} R_0$
\begin{equation}
\label{CFe}
\sum_{R\subset U', R\in \cF  } (\mu(R))^2 \le C' \mu(U')\,.
\end{equation}
\end{prop}
\begin{proof}
Just consider $\vf ={\bf 1}_{U'}$ and use \eqref{Fe} with this $\vf$.
\end{proof}

\begin{remark}
It has been proved in \cite{AMPS}, \cite{AHMV} that both claims in this proposition are, in fact, equivalent if $\cF=\cD$ (and of course we mean here that no constant is $N$-dependent).
However, for other families $\cF$ the sufficiency of \eqref{CFe} for the embedding \eqref{Fe} is not clear. We believe that the equivalence is true for prunned bi-trees. In fact, may be \eqref{CFe} is always sufficient for the embedding \eqref{Fe}. But notice that {\it always} ``no Carleson condition \eqref{CFe}" means ``no embedding".
\end{remark}

\bigskip

\begin{defin}
Condition \eqref{CFe} is called $\cF$-{\it Carleson condition} or simply {\it  Carleson condition}  if it is clear what $\cF$ is meant.
\end{defin}

Carleson conditions of \eqref{CFe} have sub-conditions, called {\bf box conditions}. Here it is:
\begin{equation}
\label{BFe}
\forall R_0\in \cD, \quad \sum_{R\in \cD(R_0),  R\in \cF  } (\mu(R))^2 \le C' \mu(R_0)\,.
\end{equation}

This is precisely bi-parameter box condition \eqref{2Pbox} if $\cF$

\bigskip

\begin{remark}
For the dyadic intervals and corresponding simple  trees the box condition is equivalent to Carleson condition and, thus, equivalent to Carleson embedding, see \cite{AHMV}.
The feeling is that Carleson condition is always strictly  stronger than box condition if one works with dyadic rectangles rather than dyadic intervals (for non-trivial families $\cF$ of dyadic rectangles). This intuitively correspond to Carleson example related with Chang--Fefferman's theorem. See \cite{Carleson}, \cite{ChF}.  Here we prove  the following theorem.
\end{remark}

\medskip

\begin{theorem}
\label{ex}
There is a family $\cF$ that generates a cut bi-tree $\cT_\cF$ (which is also a pruned bi-tree) such that the box condition \eqref{BFe} is satisfied, but the Carleson condition
\begin{equation}
\label{CFe0}
\sum_{R\subset \Omega, R\in \cF  } \mu(R)^2 \le C' \mu(\Omega)\,.
\end{equation}
fails for a particular $\Omega=\cup_{R_0\in \cF'} R_0$ for some $\cF'\subset \cF$. 
\end{theorem}
In particular, there is no embedding \eqref{Fe} for this family $\cF$ even though the box condition \eqref{BFe} holds.

\bigskip

\subsection{Box conditions versus Carleson condition}
We wish to construct a measure $\mu>0$ on the unit square $Q$ such that the box condition is satisfied
\begin{equation}
\label{box0}
\sum_{R\in \cD(R_0)} \mu(R)^2 \, \al_R\le C_0\, \mu(R_0), \quad \forall R_0\in \cD,
\end{equation}
where $\cD$ is all dyadic sub-rectangles of the unit  square, $\cD(R_0)$ is all dyadic sub-rectangles of $R_0$.
But Carleson condition is not satisfied, meaning that there is a collection $\cR$ of rectangles
of $\cD$, such that if 
\begin{equation}
\label{UcR}
U_\cR= \cup_{R\in \cR} R,
\end{equation}
then
\begin{equation}
\label{Carl}
\sum_{R\subset U_\cR} \mu(R)^2\, \al_R \ge C_1 \,\mu(U_\cR),
\end{equation}
where $C_1$ is as large as we wish with respect to $C_0$.

\medskip


Let $\cF$ be $\cD\setminus {\mathcal H}$, where  ${\mathcal H}$ is the set of all hooked dyadic rectangles inside $Q=[0,1]^2$ such that their ``right upper corner lies above hyperbola" (see the definitions below).  Let
$$
\al_R =\begin{cases} 1,\,\, \text{if} \,\, R\in \cF,
\\
0,\,\, \text{if} \,\, R \in {\mathcal H}\,.
\end{cases}
$$
Then we construct measure $\mu>0$ on $Q$ such that for
\begin{equation}
\label{boxH}
\sum_{R\in \cD(R_0)} \mu(R)^2\, \al_R = \sum_{R\in \cD(R_0), R\notin {\mathcal H}} \mu(R)^2 \le C_0 \,\mu(R_0), \quad \forall R_0\in \cD,
\end{equation}
but Carleson condition is not satisfied, meaning that there is a collection $\cR$ of dyadic rectangles
(they all will be from $\cF=\cD\setminus{\mathcal H}$), such that if 
\begin{equation}
\label{UcRH}
U_\cR= \cup_{R\in \cR} R,
\end{equation}
then
\begin{equation}
\label{CarlH}
\sum_{R\subset U_\cR} \mu(R)^2\, \al_R = \sum_{R\subset U_\cR, R\notin {\mathcal H}} \mu(R)^2 \ge C_1\, \mu(U_\cR),
\end{equation}
where $C_1$ is as large as we wish with respect to $C_0$.

This construction (see below) of course proves Theorem \ref{ex}.

\bigskip

\section{Comparison with Carleson's counterexample related to Chang--Fefferman's product $BMO$}
\label{Carl-sect}

In \cite{Carleson} Carleson constructed the families $\cR$ (there is a sequence of families, but we skip the index) of dyadic sub-rectangles of $Q=[0,1]^2$ having the following two ``contradicting" properties:
\begin{equation}
\label{boxCa}
\forall R_0 \in \cD, \quad \sum_{R\subset R_0, R\in \cR} m_2(R) \le C_0 m_2(R_0)\,,
\end{equation}
but for $U_\cR:= \cup_{R\in \cR} R$
\begin{equation}
\label{Ca}
 \sum_{ R\in \cR} m_2(R) \ge C_1 m_2(U_\cR)\,,
\end{equation}
where $C_1/C_0$ is as big as one wishes.

Suppose one would be able to put square over $m_2(R)$ in the LHS of \eqref{Ca}. Then the counterexample of the type \eqref{box0} + \eqref{Carl} would immediately follow with $\al_R\equiv 1$. Explanation: it is just a trivial fact that (notice the summation over {\it all} dyadic rectangles inside $R_0$ below)
$$
\forall R_0 \in \cD, \quad \sum_{R\subset R_0, R\in \cD} [m_2(R)]^2 \le 4 m_2(R_0)\,.
$$
This and \eqref{boxCa} would give us box condition with $\al_R\equiv 1$. On the other hand, \eqref{Ca} (with square in the LHS!) would give non-existence of the embedding.

\medskip

However, all this is wishful thinking.  It is impossible to have the analog of \eqref{Ca}  but with
summation of $[m_2(R)]^2$ in the left hand side of \eqref{Ca}. Indeed, in \cite{Carleson} the family $\cR$ has the property that $\sum_{R\in \cR} m_2(R) =1$.
Therefore,
$$
 \sum_{ R\in \cR} [m_2(R)]^2 \le \max_{R\in \cR} m_2(R)  \le 1\cdot m_2(U_\cR)\,,
 $$
 and this is exactly opposite of \eqref{Ca}.

\medskip

\subsection{Carleson's construction gives counterexample with wild $\{\al_R\}_{R\in \cD}$}
\label{wild}

We failed to derive the counterexample with $\al_R\equiv 1$ from Carleson's construction \cite{Carleson}. But with {\it some} (rather wild) $\{\al_R\}_{R\in \cD}$ such counterexample is readily provided by the collection $\cR$ of dyadic rectangles constructed by Carleson.

Let us put
$$
\al_R= \begin{cases} \frac{1}{m_2(R)}, \quad R\in \cR,
\\
0, \quad \text{otherwise}
\end{cases}
$$
Measure $\mu$ is just $m_2$. Then fix any dyadic rectangle $R_0$, then box condition \eqref{box0}:
$$
\sum_{R\subset R_0}\mu(R)^2 \al_R = \sum_{R\subset R_0, R\in \cR} m_2(R) \le m_2(R_0)=\mu(R_0)\,.
$$
But if $\Omega:= \cup_{R\in \cR} R$, then
$$
\sum_{R\subset \Omega}\mu(R)^2 \al_R =\sum_{R\subset \Omega, R\in \cR} m_2(R) =1\ge C\,m_2(\Omega)\,,
$$
where $C$ can be chosen as large as one wants. Hence, \eqref{Ca}  holds too with large constant.

\section{The dual formulation of embedding}
\label{dualForm}
We say that vertices $\gamma_2>\gamma_1$, $\gamma_1, \gamma_2\in \cT_\cF$ if the dyadic rectangle $R(\gamma_2)$ (to which $\gamma_2$ corresponds) contains the dyadic rectangle $R(\gamma_1)$ (to which $\gamma_1$ corresponds). With this partial order on vertices of bi-tree $\cT_\cF$ we can define $A_\gamma$--the set of ancestors of $\gamma$ and $S_\gamma$--the set of successors of $\gamma$.
Let $\psi$ be a non-negative function on vertices of bi-tree $\cT_\cF$. Define for every $\gamma \in \cT_\cD$
$$
\bI_\cF \psi (\gamma)= \sum_{\gamma'\ge_\cD \gamma, \gamma' \in \cT_\cF} \psi(\gamma'),
$$
It is Hardy operator  on the full bi-tree $\cT_\cD$ but defined using $\cT_\cF$.

We have to explain what is $\ge_\cD$. If we write $\gamma'\ge_\cD \gamma$, $\gamma \in \cT_\cD, \gamma' \in \cT_\cF$, this means that $\gamma'\ge \gamma$ in $\cT_\cD$, but it is possible that there is no path down from  $\gamma'$ to $\gamma$ in a smaller bi-tree $\cT_\cF$.

\medskip

Let us make another look at the  bi-trees Carleson embedding:
\begin{equation}
\label{CEF}
\forall \vf,\quad  \sum_{R\in \cF} (\int_R \vf d\mu)^2 \le C_0 \int_Q \vf^2 d\mu
\end{equation}

The dual formulation is in the following proposition about weighted Hardy inequality on bi-trees.

\begin{prop}
\label{dual}
Embedding \eqref{CEF} is equivalent to the following statement:
\begin{equation}
\label{CEFdual}
\forall \psi \in \ell^2(\cT_\cF, \al),\quad  \sum_{\om} \bI_\cF \psi (\om)^2 \mu(\om) \le  C_0 \|\psi\|_{\ell^2(\cT_\cF, \al)}^2:= C_0 \sum_{\gamma\in \cT_\cF} \psi(\gamma)^2 \al_{R(\gamma)}\,.
\end{equation}
\end{prop}

\bigskip

\subsection{The comparison with box condition from \cite{AMPS}}
The reader should be warned that there are many absolutely different box conditions. For example, in \cite{AMPS} the box condition looks as follows:
$$
\forall R_0 \in \cD, \quad \mu(R_0) \le C_0 \,cap(R_0)\,.
$$
Here capacity is the one introduced in \cite{AMPS} on a full bi-tree.  For the box $R_0$ of size $2^{-m}$ by $2^{-k}$ capacity is equivalent to $1/(mk)$.

It is remarked in \cite{AMPS} that this box condition can be readily reconciled with
$$
\mu(U) \ge C_1 \,cap(U),
$$
where $U$ is a  union of  dyadic rectangles and $C_1/C_0$ is as large as one wants. This claim is  an  easy one a) because capacity is not additive, only sub-additive, b) because this claim holds even for dyadic intervals, not dyadic rectangles! 

In other words, this effect that  the capacitary box condition  in \cite{AMPS} does not imply the capacitary condition in \cite{AMPS} is true even for the simple tree; one does not need complicated combinatorics and bi-trees for this effect to take place.

\medskip

However, in the case of the simple tree our box condition \eqref{box} applied to dyadic intervals (instead of dyadic rectangles) always implies the Carleson condition \eqref{Carl}. This is obvious (see also  \cite{AHMVt}). So the effect that box condition does not imply Carleson condition in our sense  is a much more subtle effect, and one should  be able to observe it on bi-trees or more complicated graphs but not on simple trees.

What we are doing below is constructing a certain bi-tree $\cT_\cF$ on which we indeed observe this effect of having the box condition but not the Carleson one.

\bigskip

\section{Another counterexample: $\al_R=1\,\text{or}\, 0$. The construction of $\Omega$}
\label{om}

We construct in this section a special domain $\Omega$ in the unit square $Q$. For that we construct a special family of dyadic sub-rectangles of $Q$, each having the origin $(0,0)$ as its lower left corner, and then $\Omega$ will be their union. All our dyadic rectangles below shall be $N$-coarse.

\subsection{$A$-unbalanced systems}

Let us fix a dyadic square $A$ of size $2^{-N}\times 2^{-N}$ inside the unit square.
Let also $\cR=\{ R_0,\dots R_M\}$ be a family of dyadic rectangles in $Q$. We denote their union by $U_{\cR} := \cup_{R\in\cR}R$.
 We then associate two special numbers with the collection $\cR$:
	\begin{itemize}
	\item $F_A=F_A(\cR)$ ($F$ stands  for family) is the number of dyadic rectangles containing $A$ and contained in $U_\cR$;
	\item $B_A= B_A(\cR)$ ($B$ stands  for box)  is the maximal number of dyadic rectangles containing $A$ and contained in one of dyadic rectangles $R$ such that $R\subset U_\cR$.
	\end{itemize}
Note that if each $R_i$ is {\it maximal} by inclusion among dyadic rectangles contained in $U_\cR$, then to compute $B_A$ we need to compute maximum in $i$, $0\le i\le M$, of the number of dyadic rectangles containing $A$ and contained in $R_i$.

We call the family $\cR$ an \textbf{$A$-unbalanced system} if 
	\begin{equation}
	\label{unb}
	F_A \ge c\log N\,,
	\end{equation}
where $c$ is a small positive absolute constant. 
We are interested in unbalanced families with very large $N$. Clearly $F_A \le (M+1) B_A$, so the number of rectangles in the family should grow with growing $N$.

\bigskip

\subsubsection{A balanced family.}
Let us first give the example of a large and ``complicated" family, which is balanced (so it is not what we want).
Let $A$ be the left lower corner $2^{-N}\times 2^{-N}$ dyadic square. Let $\cR$ consist of
$$
R_0= [0,2^{-N})\times [0,2^0), \ldots, R_k = [0,2^{-N+k})\times [0,2^{-k}), \ldots, 
	R_N= [0,2^0)\times [0,2^{-N})\,.
$$
It is easy to compute (see Figure \ref{fig:Bal1}) that for this family $F_A\asymp N^2, B_A\asymp N^2$. In fact, one can compute these numbers exactly by a simple counting argument:
	$$F_A(\cR) = \frac{(N+1)(N+2)}{2}; \:\:\: B_A(\cR) = \left\{\begin{array}{ll}
		(N+2)^2/4, & \text{ if } N \text{ is even};\\
		(N+1)(N+3)/4, & \text{ if } N \text{ is odd.}
	\end{array}\right.$$
\medskip

However, in preparation for the next situation where it will no longer be so easy to compute $F_A$ and $B_A$, let us make another picture that shows why this happens. This construction only really sees rectangles which contains the origin, and these will continue to play an important role in the rest of this paper: 
	\begin{defin} \label{D:hooked}
	Any dyadic rectangle in $Q$ whose lower left corner is the the origin will be called a \textbf{hooked} rectangle. In our case, any $N$-coarse dyadic hooked rectangle must be of the form 
		\begin{equation}\label{E:hookedDef}
		R = [0,2^{-N+m})\times [0,2^{-N+k}).
		\end{equation}
	\end{defin}
	
	\medskip
	
	It is easy to see that inside any hooked rectangle of the form \eqref{E:hookedDef} we have $(m+1)(k+1)$ dyadic rectangles containing $A$ -- in other words, $(m+1)(k+1)$ hooked rectangles. Let us associate the pair of integers $(m, k)$ with such an $R$.
In the previous example, the family $\cR$ gets the following associated pairs:
$$
(0, N), (1, N-1), \dots, (k, N-k),\dots, (N, 0)\,.
$$

If we plot these integer points on the graph $(x, y)$ they will lie on a straight line $y=N-x$, as in Figure \ref{fig:Bal2}.
The area of the triangle {\it below} that line (and bounded by the axes) is $\approx N^2$. 
It is easy to see that, in this setting, $F_A(\cR)$ \textit{is the same as the number of integer points inside and on the boundary of this triangle}, which is comparable to the \textit{area} of the triangle. So this reasoning would give us exactly that $F_A \asymp N^2$.

Similarly, for each $R_k$, the number of hooked rectangles contained in $R_k$ corresponds now in the picture on the right to the number of integer points inside and on the boundary of $R_k$. So, $B_A(\cR)$ \textit{is proportional to the area of the largest rectangle in $(x, y)$ with one  vertex of the type $(k, N-k)$, another vertex at $(0,0)$, and with sides parallel to the axes}. This is also $\approx N^2$ (just take $k=\frac{N}2$).

Now we know what to do to build the unbalanced system.

\begin{figure}[h!]
  \begin{subfigure}[b]{0.49\textwidth}
    \includegraphics[width=\textwidth]{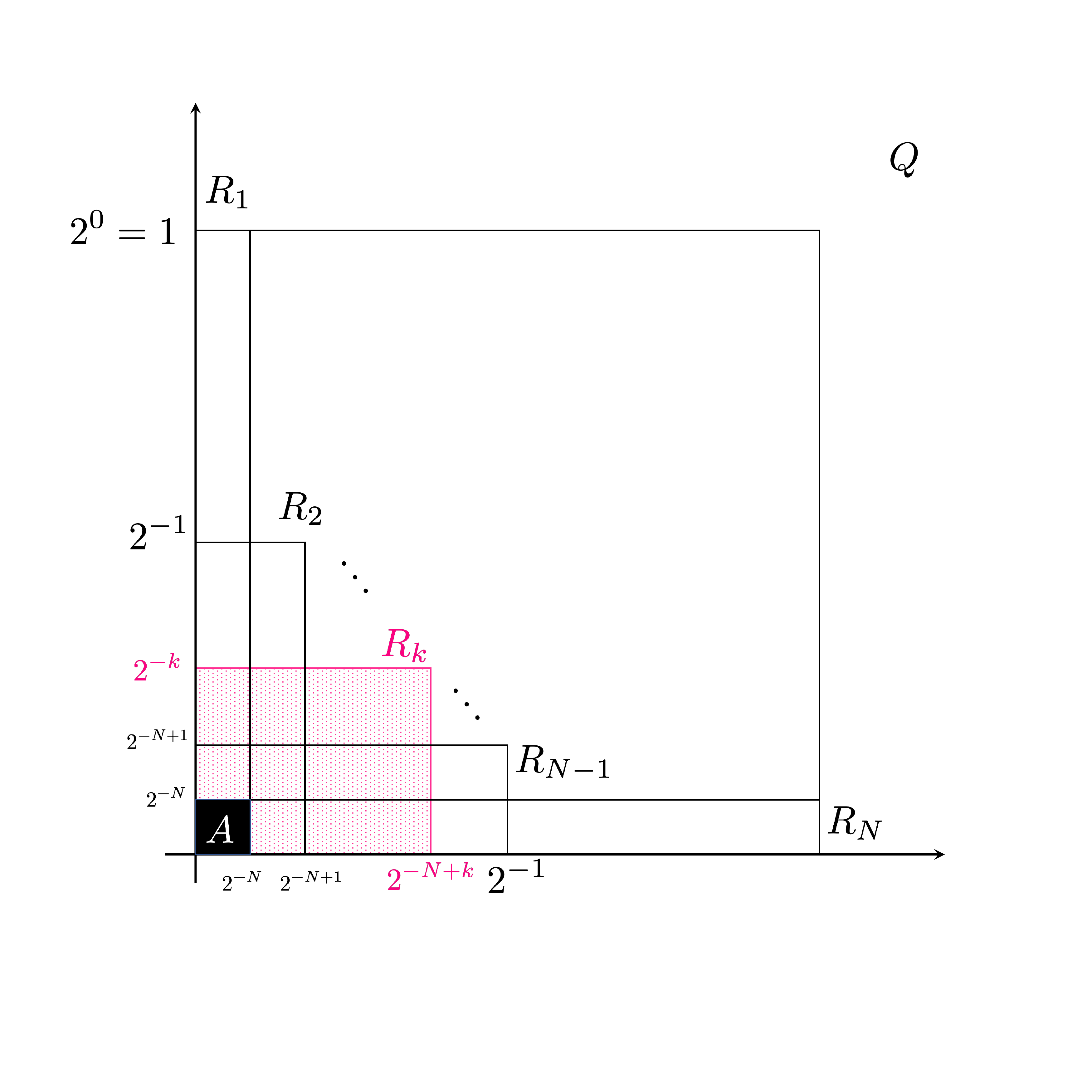}
    \caption{The $A$-balanced collection $\cR$.}
    \label{fig:Bal1}
  \end{subfigure}
  \begin{subfigure}[b]{0.49\textwidth}
    \includegraphics[width=\textwidth]{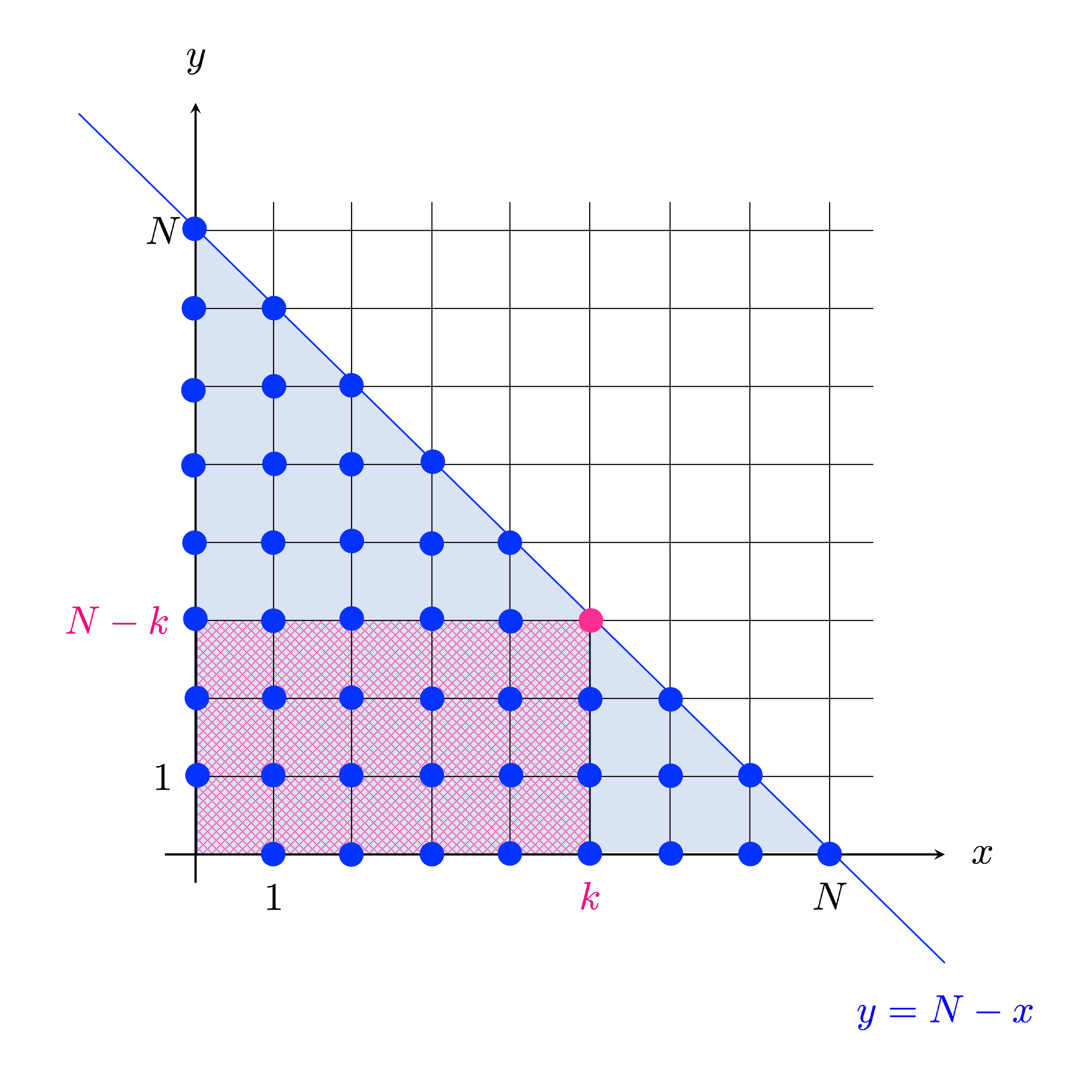}
    \caption{The figure corresponding to $\cR$ in the $(x,y)$ plane.}
    \label{fig:Bal2}
  \end{subfigure}
  \caption{A contrast of $[0,1]$ and $[0,N]$.}
  \label{fig:tess1dim}
\end{figure}

\medskip

\subsubsection{The unbalanced family.}
Let us consider the family $\cR$ of dyadic rectangles  attached to the lower left corner and having the size $2^{-N+m}\times 2^{-N+k}$, as before, but where we do not have $m+k\le N$, but rather have
$$
mk\le N.
$$
In other words, $\cR$ consists of all \textit{hooked rectangles which lie ``under the hyperbola''} $xy=N$. 

We have a similar right picture also, in the $(x, y)$ plane, by associating to each hooked $R$ the integer point $(m, k)$.
But now, instead of the area of the triangle above, we need to compare ``almost the area" under the hyperbola (we explain this term below)
 $$
 y=\frac{N}{x},\quad 1\le x\le N
 $$
and the area of the largest rectangle of the type $[0,m]\times [0, k]$ with $mk\le N$ which lies under this hyperbola. The latter area is of course $N$, so 
	$$B_A\asymp N.$$  
But the area under the hyperbola is of the order
 $$
 \int_1^N \frac{N}{x} dx\asymp N\log N\,.
 $$
The subtle point is that we need a smaller quantity, denoted earlier by the words ``almost the area" under hyperbola. These words stand for the {\it union} $u(N)$ of rectangles of the type $[0,m]\times [0, k]$ with $mk\le N$ lying entirely in the domain below hyperbola.

\bigskip

Consider the measure $\mu_A$ which is just uniform measure on $A$ with total mass $a$. The box condition for $\cR$:
	$$\sum_{R\subset R_0} \mu_A(R)^2 \leq \mu_A(R_0), \:\: \forall R_0 \in \cR$$
becomes in this case
	$$a^2\big(\#\text{ of hooked rectangles inside } R_0\big) \leq a, \forall R_0 \in \cR,$$
which will be satisfied if and only if
	$$a^2 B_A \leq a.$$
So, we may choose
	$$a=B_A^{-1} = N^{-1}.$$
	
Now, the Carleson condition would require of $\Omega = \cup_{R\in\cR}R$ to satisfy
	$$
	\sum_{R\subset \Omega} \mu(R)^2 \leq C\mu(\Omega), \text{ or  simply } a^2 F_A \leq C a.
	$$
As before, to compute $F_A$ we must really just count the hooked rectangles in $\cR$ -- for each such hooked rectangle, there is a corresponding integer point $(m, k)$ in the $(x, y)$ plane. So computing $F_A$ is the same as \textit{counting the number of integer points under or on the hyperbola} $xy=N$, which is the same as computing the area $u(N)$ mentioned earlier. As the Lemma \ref{almarea} below shows, however, we have
	$$F_A \asymp u(N) \geq cN\log N$$
for some universal constant $c$. We obtain then that
$$
c \log N \le \frac{F_A}{B_A} \le C\,,
$$
which gives a contradiction when $N\to\infty$.

 \begin{lemma}
 \label{almarea}
 For every positive integer $N$, let $u(N)$ denote the area  of the  union of rectangles of the type $[0,m]\times [0, k]$ with $m$, $k$ positive integers such that $mk\le N$, that is, lying entirely in the domain below the hyperbola $xy = N$. Then $ u(N) \geq c\int_1^N \frac{N}{x} dx$ with a small absolute positive constant $c$.
  \end{lemma}
  
  \begin{proof}
  Suppose first that $N = 2^k$ for some integer $k \geq 1$. Then $u(N)$ is certainly larger than the area of the union of rectangles of the type 
  $[0, 2^i] \times [0, 2^{k-i}]$, with $0 \leq i \leq k$. This latter area can be easily computed, and it is $2^k + k 2^{k-1}$. So
  	$$u(N) \geq 2^k + k 2^{k-1} \geq \log_2(N)\frac{N}{2} = \frac{1}{2\log 2} (N \log N), \:\:\:\:\forall N = 2^k.$$
  Now suppose $2^k < N < 2^{k+1}$. Then
  	\begin{eqnarray*}
	u(N) \geq r(2^k) &\geq& 2^k + k 2^{k-1}\\
		&>& \frac{N}{2} + \log_2\left(\frac{N}{2}\right)\frac{N}{4}\\
		&=& \frac{N}{4} + \frac{N}{4}\log_2 N \geq \frac{1}{4\log 2}(N \log N).
	\end{eqnarray*}
	So, we may take $c = \frac{1}{4 \log 2}$.
  \end{proof}
 
\begin{defin}
\label{mH}
 For every positive integer $N$, consider rectangles of the type $[0,m]\times [0, k]$ with $m$, $k$ positive integers such that $mk> N$, that is, having their upper right vertex above the hyperbola $xy = N$. We call the family of corresponding hooked rectangles  $[0, 2^{-N+m}]\times [0, 2^{-N+k}]$  by symbol ${\mathcal H} :={\mathcal H}_N$.
 \end{defin}
 
\medskip

\section{Gathering things together}
\label{together}

Counterexample should consists of measure $\mu>0$ on the unit square such that the box condition is satisfied
\begin{equation}
\label{box1}
\sum_{R\in \cD(R_0), R\notin {\mathcal H}} \mu(R)^2 \le C_0\, \mu(R_0), \quad \forall R_0\in \cD
\end{equation}
where $\cD$ is all dyadic sub-rectangles of the unit  square, $\cD(R_0)$ is all dyadic sub-rectangles of $R_0$.
But Carleson condition is not satisfied, meaning that there is a collection $\cR$ of rectangles
of $\cD$, such that if 
\begin{equation}
\label{UcR1}
U_\cR= \cup_{R\in \cR} R,
\end{equation}
then
\begin{equation}
\label{Carl1}
\sum_{R\subset U_\cR, R\notin {\mathcal H}} \mu(R)^2 \ge C_1 \mu(U_\cR),
\end{equation}
where $C_1$ is as large as we wish with respect to $C_0$.

\medskip

We fix a large $N$, measure will be $2^{-N}$-grained, meaning to have constant (or almost constant) density on each $2^{-N} \times 2^{-N}$ dyadic square. Constant $C_0$ will be
absolute, and $C_1$ will grow as $\log N$.

\medskip

We already constructed the correct family $\cR$ above. And part of the measure $\mu_A$: it was just mass $a=1/N$ on lower left $2^{-N} \times 2^{-N}$ dyadic square. 

We checked \eqref{Carl} and we checked \eqref{box1}--but only for $R_0$ in $\cR$. We need to check this for all dyadic $R_0\in \cD$.

But if $R_0$ is not in the family $\cR$ then only two things may happen: 1) $\mu(R_0)=0$,
2) $R_0\in {\mathcal H}$.

In the first case \eqref{box} is obvious. Let us see what happens in the second case, when $R_0$ is a hooked rectangle, see Definition \eqref{D:hooked}.
Let 
\begin{equation*}
		R = [0,2^{-N+m})\times [0,2^{-N+k}).
		\end{equation*}

	\medskip
	
	It is easy to see that inside any hooked rectangle of this form we have $(m+1)(k+1)$ dyadic rectangles containing $A$ -- in other words, $(m+1)(k+1)$ dyadic rectangles. 
	And all of them will be hooked.
	
	But only $\approx N$ of them will have $mk \le N$, in other words only $\approx N$ of them will be $\notin {\mathcal H}$.
	
	So inequality \eqref{box1} becomes (only hooked rectangles carry non-zero measure)
	$$
	\sum_{R\, hooked,\, R\subset R_0, R\notin H}\mu_A(R)^2 \asymp \frac{1}{N^2}\cdot (m+1)(k+1)\lesssim \frac1{N} = \mu_A(R_0)\,.
	$$

\section{Box condition  implies box capacitary condition}

In \cite{AHMV} it is shown how Carleson condition on bi-tree implies capacitary box condition on bi-tree. Here we show that the box condition alone on bi-tree  implies  capacitary box condition on bi-tree.

\medskip

We deal with dyadic rectangles inside the unit square $Q$.  We call the family of such rectangles with sides at least $2^{-N}$ by symbol $\cD:=\cD_N$.Capacity of rectangle $R= H\times V$ ($H, V$ are dyadic intervals) is
\begin{equation}
\label{R}
cap (R) = \frac{1}{\log_2 \frac1{|H|}\cdot \log_2\frac1{|V|}}\,.
\end{equation}
If we consider dyadic intervals instead of rectangles, then $cap H=\frac{1}{\log_2 \frac1{|H|}}$. This is not the definition of capacity. The definition is using the bi-tree model of dyadic rectangles (tree for dyadic intervals).

Let $\psi$ be a non-negative function on bi-tree such that
$$
\bI\psi(\om) \ge 1\quad \forall \om \subset R\,.
$$
These are ``$R$-admissible" functions. Among admissible $\psi$ choose the one such that
$$
\|\psi\|_{{\ell^2(bi-tree)}}^2 =\sum_{\gamma\in bi-tree} \psi(\gamma)^2
$$ is the smallest. Then this smallest $\|\psi\|_{{\ell^2(bi-tree)}}^2 $ is called the bi-tree capacity.
Absolutely analogous definition is for the simple tree capacity.


\medskip


Now let $R=\om=$ $2^{-N}\times 2^{-N}$ small dyadic square (let it be with the left lower corner at $(0,0)$, but this does not matter at all).
 Clearly $(N+1)^2\asymp N^2$ dyadic rectangles contain this $\om$. So this is how many ancestors on $\om$ we will have in a bi-tree. Now we need to distribute positive numbers on these ancestors to have the sum of these numbers to be equal to $1$, but the sum of square of these numbers to be as small as  possible.

Obvious choice is to give each ancestor $\gamma$ the same $\psi(\gamma)\asymp \frac1{N^2}$.
Then
$$
\|\psi\|_{{\ell^2(bi-tree)}}^2 =\sum_{\gamma\in bi-tree} \psi(\gamma)^2 \asymp \Big(\frac1{N^2}\Big)^2 N^2 = \frac1{N^2}\,.
$$
This proves that $cap \,\om \asymp  \frac1{N^2}$. By the same logic we see \eqref{R}, or, the capacity of dyadic interval  to be $cap H=\frac{1}{\log_2 \frac1{|H|}}$ if we consider the simple tree.

\bigskip

In \cite{AMPS} it is proved that 
\begin{equation}
\label{muR}
\mu(R) \le \tilde C\, cap (R) = \frac{\tilde C}{\log_2 \frac1{|H|}\cdot \log_2\frac1{|V|}}
\end{equation}
is the necessary condition for Carleson embedding on bi-tree (and tree of course).
In particular,
\begin{equation}
\label{muom}
\mu(\om) \le \frac{\tilde C}{N^2}
\end{equation}
is necessary for Carleson embedding. So ``no \eqref{muom}"-- then ``forget about embedding".


\bigskip

In fact, \cite{AMPS} proves that capacitary condition of type \eqref{muR} is necessary and {\it sufficient} for the Carleson embedding if to allow not just dyadic boxes in \eqref{muR} but any unions of dyadic boxes.

We can call \eqref{muR}, \eqref{muom} ``box capacitary conditions".


\bigskip

But now we will show  \eqref{muom} in ``a simple way". To derive it we will not be using the full strength of Carleson embedding, but only box condition as follows:
\begin{equation}
\label{box}
\forall R_0\in \cD, \quad \sum_{R\subset R_0, R\in \cD} \mu(R)^2 \le C_0 \mu(R_0)\,.
\end{equation}
By multiplying $\mu$ by $C_0^{-1}$ we can always think that $C_0=1$.

So let $\om$ is $2^{-N}\times 2^{-N}$ dyadic square,  let it be with the left lower corner at $(0,0)$ (but this does not matter at all). Let 
$Q_{m, k}$ be a hooked rectangle inside $Q$ with upper right corner at $(2^{-N+m}, 2^{-N+k})$.
Only $\{Q_{m, k}\}_{m,k}$, $0\le m\le N, 0\le k\le N$, contain $\om$.

Let 
 $$
4a^2\ge \mu(\om)= \mu([0, 2^{-N}]^2) \ge a^2\,.
 $$

Let us denote $\mu(m, k) = \mu (Q_{m,k})$. If we fix $(m, k)$, then $Q_{m, k}$ contains at least $m\cdot k$ other $Q_{m', k'}$ that contain $\om$. Let us use this fact, and
 apply \eqref{box}  (with $C_0=1$) for the first time:
\begin{equation}
\label{mk}
 \mu(m, k) \ge a^4\, mk\,.
 \end{equation}

Now put $a_0:=a, a_1:= 2a,\dots, a_\ell := 2^\ell a$, and we wish to choose $k_1$ in such a way that
$$
\mu(k_1, k_1) \ge k_1^2 a_0^4 \ge 4a_0^2=a_1^2 \,.
$$

By \eqref{mk} we choose $k_1 =\frac{2}{a_0}=\frac{2}{a}$.

But by a the same reasoning as \eqref{mk} we get
\begin{equation}
\label{k1k2}
\mu(k_1+k_2, k_1+k_2) \ge k_2^2 \,\mu(k_1, k_1) \ge k_2^2\, a_1^4\,.
\end{equation}
We wish to choose $k_2$ in such a way that
$$
\mu(k_1+k_2, k_1+k_2) \ge a_2^2 =4a_1^2\,.
$$
As before we choose $k_2 =\frac{2}{a_1}= \frac{2}{2a_0}=\frac{1}{a}$.
We continue in the same vein. We next use
\begin{equation}
\label{k1k2k3}
\mu(k_1+k_2+k_3, k_1+k_2+k_3) \ge k_3^2\, \,\mu(k_1+k_2, k_1+k_2)= k_3^2\, a_2^4
\end{equation}
And we choose $k_3$ to have
$$
\mu(k_1+k_2+k_3, k_1+k_2+k_3,) \ge 4a_2^2 = a_3^2.
$$
As before we choose $k_3 =\frac{2}{a_2}= \frac{2}{2^2a}=\frac{1}{2a}$.
 
 
 We continue this recursion. Without loss of generality we can think that $a=2^{-s}$, and the recursion holds while $k_j\ge 1$. But
 $$
 k_\ell = \frac{1}{2^{\ell-2} a} = \frac{2^s}{2^{\ell-2}}\,.
 $$
 So let us end recursion  when $\ell = s $. Then
 $$
 \mu(k_1+k_2+k_3+\dots +k_\ell, k_1+k_2+k_3+ \dots +k_\ell) \ge  a_\ell^2 =2^{2\ell} a^2=1\,.
$$
Suppose that
\begin{equation}
\label{N2}
k_1+\dots +k_\ell  = \frac{2}{a}\big(1+\frac12+\frac1{2^2}+\dots+\frac1{2^\ell}\big) \le N/2\,.
\end{equation}
Then we can make one more step. Fix some $k\le N/2$ and write
$$
 \mu(k_1+k_2+k_3+\dots +k_\ell+k, k_1+k_2+k_3+ \dots +k_\ell+k) \ge k^2  \,\, \mu(k_1+\dots +k_\ell, k_1+\dots +k_\ell) \ge k^2\cdot 1\,.
 $$
 This is impossible since $\mu([0,1]^2) \le 1$ and square $[0, 2^{-N + k_1+k_2+k_3+\dots +k_\ell+k}]^2$ is inside $[0,1]^2$.
 Hence the negation of \eqref{N2} holds, so
 $a\le cN^{-1}$. Thus, $\mu([0, 2^{-N}]^2) \le c'N^{-2}$. By the same reasoning one proves that
 $$
 \mu([0, 2^{-m}]\times [0, 2^{-k}]) \le c\, (mk)^{-2}\,.
 $$

\end{document}